\documentclass[12pt]{amsart}
\usepackage{amssymb,latexsym}
\usepackage{amsfonts}
\usepackage{amsmath}
\usepackage{mathrsfs}
\usepackage[colorlinks,linkcolor=blue,anchorcolor=blue,citecolor=blue]{hyperref}
\usepackage{algorithm}
\usepackage{enumerate}
\usepackage{algpseudocode}
\usepackage{verbatim}
\usepackage{graphicx}

\newcommand\C{{\mathbb C}}
\newcommand\Q{{\mathbb Q}}
\newcommand\R{{\mathbb R}}
\newcommand\Z{{\mathbb Z}}

\newcommand\al{\alpha}
\newcommand\be{\beta}
\newcommand\ga{\gamma}
\newcommand\eps{\epsilon}
\newcommand\PP{{\bf P}}

\newcommand\rA{\mathscr{A}}
\newcommand\B{\mathscr{B}}

\newcommand\D{\mathscr{D}}

\newcommand\E{\mathscr{E}}
\newcommand\rP{\mathscr{P}}

\newtheorem{theorem}{Theorem}[section]
\newtheorem{lemma}[theorem]{Lemma}

\newtheorem{corollary}[theorem]{Corollary}

\theoremstyle{definition}

\theoremstyle{remark}

\numberwithin{equation}{section}

\def\({\left(}
\def\){\right)}
\def\[{\left[}
\def\]{\right]}
\def\<{\langle}
\def\>{\rangle}

\begin{document}

\title[Positive density of integer polynomials]{Positive density of integer polynomials with some prescribed properties}

\author{Art\= uras Dubickas}
\address{Department of Mathematics and Informatics, Vilnius University, Naugarduko 24,
LT-03225 Vilnius, Lithuania}
\email{arturas.dubickas@mif.vu.lt}

\author{Min Sha}
\address{School of Mathematics and Statistics, University of New South Wales,
 Sydney, NSW 2052, Australia}
\email{shamin2010@gmail.com}





\begin{abstract}
In this paper, we show that various kinds of integer polynomials with prescribed properties of their roots have positive density. For example, we prove that almost all integer polynomials have exactly one or two roots with maximal modulus. We also show that for any positive integer $n$ and any set of $n$ distinct points symmetric with respect to the real line, there is a positive density of integer polynomials of degree $n$, height at most
$H$ and Galois group $S_n$ whose roots are close to the given $n$ points.
\end{abstract}

\maketitle



\section{Introduction}

The so-called \textit{fundamental theorem of algebra}, first proved by Gauss, asserts that every non-constant univariate polynomial with complex coefficients has at least one complex root.  Polynomials and various properties of their roots often play a crucial role in studying many mathematical objects.

For example, the properties of linear recurrence sequences (LRS) rely heavily on their characteristic polynomials.
Recall that a LRS $\{u_n\}_{n=0}^{\infty}$ of order $n\ge 2$ with elements in $\C$ is defined by the linear relation
$$
u_{m+n}=a_1u_{m+n-1}+\cdots+a_nu_m \quad (m=0,1,2,\dots),
$$
where $a_1,\dots,a_n\in \C$ (called the \textit{coefficients}), $a_n\ne 0$ and $u_j \ne 0$ for at least one
 $j$ in the range $0 \le j \le n-1$.
The characteristic polynomial of this LRS is
$$
f(X)=X^n-a_1X^{n-1}-\cdots-a_n =\prod_{i=1}^{n} (X-\al_i) \in \C[X],
$$
where all $\al_i$ are said to be the \textit{characteristic roots} of the sequence $\{u_n\}_{n=0}^{\infty}$. The sequence $\{u_n\}_{n=0}^{\infty}$ is called \textit{non-degenerate} if the quotient of any of its two distinct characteristic roots is not a root of unity. Then, there are only finitely many integers $n$ with $u_n=0$ if the sequence is non-degenerate; otherwise, there may be infinitely many such integers. It has been proved in \cite{DS2014a} that almost all LRS with integer (or rational) coefficients are non-degenerate.

In this paper, a polynomial is called \textit{dominant} if it has a simple root (called \textit{dominant root}) whose modulus is strictly greater than the moduli of its remaining roots.
The LRS with dominant characteristic polynomial (called \textit{dominant LRS}) are often much easier to deal with, especially when considering Diophantine properties of linear recurrence sequences.

In order to interpret the meanings of ``almost all'' and ``positive density'', we introduce the following definition, which only shows the right way to comprehend ``almost all'' and ``positive density'', but does not explain all their meanings in this paper.
Given a proposition $\PP$ related to integer polynomials, for integers $n\ge 1$ and $H\ge 1$, we define the set
\begin{equation*}
\begin{split}
\rP_n(H) =  \{f(X) = &X^n +a_1X^{n-1} + \cdots + a_n\in \Z[X]~:\\
&  |a_i| \le H, \ i =1, \ldots, n\}.
\end{split}
\end{equation*}
We say that {\bf P} is true \emph{for almost all} (or \textit{with density tending to $1$}) monic integer polynomials
if for any integer $n\ge 1$ we have
$$
\lim_{H\to \infty} \frac{|\{f\in \rP_n(H): \textrm{{\bf P} is true for $f$}\}|}{|\rP_n(H)|}=1.
$$
Throughout, we use $|T|$ to denote the cardinality of a finite set $T$.
In addition, we say that {\bf P} is true \textit{with positive density}
if for any integer $n\ge 1$ the following holds:
$$
\limsup_{H\to \infty} \frac{|\{f\in \rP_n(H): \textrm{{\bf P} is true for $f$}\}|}{|\rP_n(H)|}>0,
$$
which is equivalent to the following
$$
H^n \ll |\{f\in \rP_n(H): \textrm{{\bf P} is true for $f$}\}| \ll H^n.
$$
In the same way we will also handle all integer polynomials (not ne\-cessarily monic), where the term $H^n$ will be replaced by $H^{n+1}$.

It has been shown in \cite{DS2014b} that almost all monic integer polynomials are dominant \cite[Theorem 1.1]{DS2014b}, but it is not true for general integer polynomials (not necessarily monic) \cite[Theorem 1.4]{DS2014b}. That is, not almost all integer polynomials (not necessarily monic) have a dominant root.
Therefore, it is natural to ask which kind of integer polynomials we should add to the set of dominant integer polynomials to make their density tend to one. The answer is ``those
with exactly two maximal complex conjugate roots".
In Section \ref{sec:modulus}, we show that almost all integer polynomials have at most two roots
with maximal modulus. Furthermore,
unlike in the monic case,
``at most two" cannot be replaced by ``one" or ``two". In particular, this result has been used in \cite{Sha} to show that for almost all LRS of algebraic numbers with rational coefficients, the \emph{Skolem problem}, which is to determine whether a given LRS has a zero term, is decidable with an explicit complexity. Besides, the result also gives a strong evidence that at least half of integer polynomials are dominant.

More precisely, put $A_n^*(k,H)$ for the cardinality of the set of integer polynomials of degree $n$ and height at most $H$ with exactly $k$ roots of maximal modulus.
The next theorem shows that the proportion of polynomials with at most two roots of maximal modulus  ($A_n^*(1,H)+A_n^*(2,H)$) among all integer polynomials of degree $n$ and height at most $H$ (there are $2H(2H+1)^n$ of them) tends to $1$
as $H \to \infty$. Namely,
\begin{theorem}\label{vienas}
For any integer $n\ge 2$, we have
$$
\lim_{H \to \infty} \frac{A_n^*(1,H)+A_n^*(2,H)}{2H(2H+1)^{n}} = 1.
$$
\end{theorem}

Although it has been proved in \cite{DS2014b} that almost all monic integer polynomials are dominant, in Section \ref{sec:modulus} we will also determine the growth rate of such monic integer polynomials with more than one root of maximal modulus.
Theorem~\ref{vienas} will be derived from Theorem~\ref{thm:moduli3+*} which estimates the difference between
the quantities $2H(2H+1)^n$ and $A_n^*(1,H)+A_n^*(2,H)$.

As an another example, let us consider number fields. Let $K$ be a number field of degree $n\ge 2$ and generated by $\al$ over $\Q$, and let $f$ be the minimal integer polynomial of $\al$ over $\Z$. If $f$ has exactly $r$ real roots and $2s$ non-real roots ($r\ge 0, s\ge 0$), then it is well known that there are $r+s$ non-isomorphic embeddings of $K$ into $\C$ ($r$ real embeddings and $s$ complex embeddings), and the rank of the unit group of the ring of integers of $K$ is $r+s-1$ (by  Dirichlet's unit theorem).
So, for two given non-negative integers $r,s$ with $n=r+2s\ge 2$, one may investigate the proportion of  number fields of degree $n$ with exactly $r$ real embeddings and $s$ complex embeddings.  We will investigate this problem in the setting of integer polynomials.
In Section \ref{sec:location} we will show that  various kinds of integer polynomials with prescribed locations of their roots have positive density; for instance, integer polynomials of degree $n$ with exactly $r$ real roots and $2s$ non-real roots, where $r+2s=n$.
In some sense, these results  suggest that there are quite many number fields with arbitrary real and complex embeddings and, in particular, their density is positive.

More precisely,
for positive integers $n,H$ and non-negative integers $r,s$ with $n=r+2s$,  let us define the set
\begin{align*}
\D_n^*(r,s;H) = & \{f(X) = a_0X^n + a_1X^{n-1}  +\cdots + a_n \in \Z[X]: H(f) \le H,  \\
& \textrm{$a_0\ne 0$, $f$ has exactly $r$ real roots and $2s$ non-real roots} \}.
\end{align*}
Then, put $D_n^*(r,s;H)=| \D_n^*(r,s;H)|$.
The following result asserts that there are many integer polynomials of degree $n$,  which have $r$ real roots and $2s$ non-real roots, where $r+2s=n$.

\begin{theorem}
\label{thm:rs}
For any integers $n, H\ge 1$ and any non-negative integers $r,s$ such that $n=r+2s$, we have
\begin{align*}
 H^{n+1} \ll D_n^*(r,s;H) \ll H^{n+1}.
\end{align*}
\end{theorem}

Recall that van der Waerden \cite{Waerden1936} showed that, for almost all   integer polynomials of degree $n\ge 1$, the Galois group of their splitting fields over $\Q$ is the full symmetric group $S_n$. Theorem~\ref{thm:rs} implies that various placings of roots are
possible for generic polynomials with Galois group $S_n$.

Theorem~\ref{thm:rs} is an immediate corollary of the following more general result:

\begin{theorem}
\label{thm:rs11}
Let $S$ be a set of $n$ distinct points $\be_1,\dots,\be_n$ in the complex plane $\C$ symmetric
with respect to the real line, i. e. $$S=\{\be_1,\dots,\be_n\}=\{\overline{\be_1},\dots,\overline{\be_n}\}.$$
Then, for each sufficiently small positive number $\varepsilon$,
there exists a constant $C=C(\varepsilon,S)>0$ such that the number
of integer polynomials $f$ of degree $n$ and height at most $H$ with
Galois group $S_n$ and
$n$ roots lying $\varepsilon$-neighbourhoods of $\be_1,\dots,\be_n$
respectively (and, in addition, such that the root of each such $f$ lying in the open disc $|z-\be_i|<\varepsilon$ is real if and only if the number $\be_i$ is real) is at least
$C H^{n+1}$.
\end{theorem}

In some sense Theorem~\ref{thm:rs11} is a non-monic version of an old theorem of Motzkin \cite{motz}; see also \cite{motz1}.

We also want to compare our results with a related result of Akiyama and Peth{\H o} \cite{AP} (see also their previous paper \cite{AP0}).
Let $M_n(r,s;T)$ be the number of monic integer polynomials of degree $n$ and having exactly $r$ real roots and $2s$ non-real roots such that the maximal modulus of the roots is less than $T$.
In \cite[Theorem 3.1]{AP}, it has been proved that
\begin{equation}
\label{eq:AP}
\left| M_n(r,s;T) - v_n^{(s)}T^{n(n+1)/2}\right| \le c_1 T^{n(n+1)/2-1},
\end{equation}
where $v_n^{(s)}$ is some quantity defined in \cite[Section 2]{AP}, and $c_1$ is a constant depending only on $n$ and $s$. This result suggests that such monic integer polynomials have positive density among all the monic integer polynomials. So, in some sense Theorem \ref{thm:rs} generalizes the result to the case of general integer polynomials (not necessarily monic). It would be of interest to obtain an asymptotic formula as \eqref{eq:AP} in our setting as well.

Throughout the paper, we use the Landau symbol $O$ and the Vinogradov symbol $\ll$. Recall that the assertions $U=O(V)$ and $U \ll V$ (sometimes we will write this also as $V \gg U$) are both equivalent to the inequality $|U|\le cV$ with some constant $c>0$.
In this paper, the constants implied in the symbols $O, \ll$ and in the phrase ``up to some constant'' only possibly depend on the degree $n$. Besides, we always assume that $H$ is a positive integer (greater than $1$ if there is the factor $\log H$ in the corresponding formula), and $n$ is a positive integer.

The proofs come from several sources. Some preliminary results and our main lemmas (Lemma~\ref{keylem} and Lemma~\ref{lem:pair}) are given in the next section. The proof of some results related to Theorem~\ref{vienas} and
several similar results are given in
Section~\ref{sec:modulus} which is concluded with two proofs of Theorem~\ref{vienas} itself.
The proof of Theorem~\ref{thm:rs11} is given at the beginning of
Section~\ref{sec:location}. The lower bounds usually come from some series of explicit examples (monic  case) and from Lemma~\ref{keylem} (non-monic case).

\section{Preliminaries}

In this section,  for the convenience of the reader we gather some definitions and give some results which will be used later on.

\subsection{Some useful results about polynomials}

Given a polynomial
$$
f(X)=a_0X^n+a_1X^{n-1}+\cdots+a_n=a_0 (X-\al_1)\cdots (X-\al_n) \in \C[X]
$$
of degree $n\ge 1$, we assume that the roots $\al_1,\dots,\al_n$
(listed with multiplicities)
are labelled so that $|\al_1 | \ge |\al_2| \ge \cdots \ge |\al_n |$. In case $|\al_1|=\cdots =|\al_k |>|\al_{k+1} |$, we say that \textit{$f$ has exactly $k$ roots with maximal modulus}.

For the above polynomial $f(X)$, its {\it height}
is defined by
$$
H(f)=\max_{0 \leq j \leq n} |a_j|,
$$
and its {\it Mahler measure} by
$$
M(f)=|a_0| \prod_{j=1}^n \max\{1,|\al_j|\}.
$$
These two quantities are related by the following well-known inequality
\begin{equation}\label{Mahler}
2^{-n}H(f)  \leq M(f) \leq \sqrt{n+1}H(f),
\end{equation}
for instance, see \cite[(3.12)]{Waldschmidt2000}.

There are various bounds on the moduli of roots of polynomials. Here, we use the one due to Fujiwara \cite{Fujiwara}; for a reference in English see \cite{Batra}.

\begin{lemma}
\label{Fujiwara}
Given a polynomial $f(X)=a_0X^n+\cdots+a_n \in \C[X]$ of degree $n\ge 1$, the moduli of all its roots do not exceed
$$
2 \max \left\{ \left|\frac{a_1}{a_0}\right|, \left|\frac{a_2}{a_0}\right|^{1/2}, \ldots, \left|\frac{a_{n-1}}{a_0}\right|^{1/(n-1)}, \left|\frac{a_n}{2a_0}\right|^{1/n} \right\}.
$$
\end{lemma}

The following result was obtained by Ferguson \cite{Ferguson1997}; see also a previous result of Boyd \cite{Boyd}.

\begin{lemma}
\label{Ferguson}
If $f \in \Z[X]$ is an irreducible polynomial which has exactly $m$ roots on a circle $|z|=c>0$, at least one of which is real, then one has $f(X)=g(X^m)$, where the polynomial $g \in \Z[X]$ has at most one real root on any circle in the plane with center at the origin.
\end{lemma}

It is well known that the roots of a polynomial depend continuously on its  coefficients. Here, we essentially quote an explicit version about this continuous dependence from \cite[Theorem 1.4]{Marden}.

\begin{lemma}
\label{lem:conroot}
Let $f,g \in \C[X]$ of degree $n\ge 1$ be as follows:
\begin{align*}
& f(X) = a_0X^n + a_1 X^{n-1} + \cdots + a_n = a_0 \prod_{i=1}^{m}(X-\al_i)^{e_i}, \\
& g(X) = (a_0+\eps_0)X^n + (a_1+\eps_1) X^{n-1} + \cdots + (a_n+\eps_n),
\end{align*}
where $\al_1,\ldots, \al_m$ are distinct roots of $f$, and $e_1+\cdots +e_m=n$. Let
\begin{align*}
&0 < \ga < \frac{1}{2}\min_{i\ne j}|\al_i - \al_j|, \quad M=
\max_{1\le j \le m} \left\{ \sum_{i=0}^{n} (\ga + |\al_j |)^i \right\} , \\
&\delta = \min_{1\le j \le m} \left\{ |a_0|\ga^{e_j} \prod_{i=1,i\ne j}^{m} (|\al_i - \al_j| - \ga)^{e_i} \right\}.
\end{align*}
(Note that if $e_1=n$, then $\ga$ can be chosen as an arbitrary
positive real number, and $\delta=|a_0|\ga^{e_1}$.)
Choose a real number
$\eps = \delta /M$, and assume that
 $|\eps_i| < \eps$ for any $0\le i \le n$. Then, $g$ has
precisely $e_j$ zeros in the disk $D_j$ with center at $\al_j$
and radius $\ga$ for each $1\le j \le m$ (counted with multiplicity).
\end{lemma}

\begin{proof}
Let
$$
h(X)= \eps_0X^n + \eps_1 X^{n-1} + \cdots + \eps_n.
$$
Then, $g=f+h$. Fix an integer $j$ with $1\le j \le m$.
Pick an arbitrary point $x$ on the boundary of $D_j$, clearly we have
$|h(x)| < \eps M$. Besides,
$$
|f(x)| \ge |a_0|\ga^{e_j} \prod_{i=1,i\ne j}^{m} (|\al_i - \al_j| - \ga)^{e_i} \ge \delta.
$$
Since $\delta =\eps M$, on the boundary of $D_j$ we always have $|h(x)| < |f(x)|$. By Rouch\'e's theorem, this means that $g$ has the same number of zeros in $D_j$ as does $f$. Note that $f$ has only one root $\al_j$ in $D_j$ with multiplicity $e_j$.
Thus, the polynomial $g$ has precisely $e_j$ zeros in the disk $D_j$.
\end{proof}

Let $\rho_m(n,H)$ be the number of monic polynomials
$$
f(X) = X^n +a_1X^{n-1} + \dots + a_n\in \Z[X], \quad n\ge 2, \quad H(f)\le H,
$$
which are reducible in $\Z[X]$ with an irreducible factor of degree $m$,
$1 \le m \le n/2$.
In \cite{Waerden1936}, van der Waerden  proved the following sharp lower and upper bounds for $\rho_m(n,H)$; see also \cite{Chela1963} for a reference in English.

\begin{lemma}
\label{reducible}
For integers $n\ge 2$ and $m\ge 1$, we have
$$H^{n-m}\ll \rho_m(n,H) \ll H^{n-m} \quad \textrm{if} \quad 1 \le m < n/2,$$
$$H^{n-m}\log H\ll \rho_m(n,H) \ll H^{n-m}\log H \quad \textrm{if} \quad  m=n/2.$$
\end{lemma}

Similarly, if $\rho_m^{*}(n,H)$ is the number of polynomials
$$
f(X) = a_0X^n +a_1X^{n-1} + \dots + a_n\in \Z[X], \quad a_0\ne 0, n\ge 2,
$$
which are reducible in $\Z[X]$ with an irreducible factor of degree $m$, $1 \le m< n/2$, and satisfy $H(f)\le H$, then, Kuba \cite[Theorems 1 and 4]{Kuba} (see also \cite{DuB}) showed that

\begin{lemma}\label{reducible2}
For integers $n \ge 3$ and $m\ge 1$, we have
$$H^{n-m+1}\ll \rho_m^{*}(n,H) \ll H^{n-m+1} \quad \textrm{if} \quad 1 \le m<n/2.$$ Furthermore, the number of reducible integer polynomials of degree $n \ge 3$ and height at most $H$ is $O(H^{n})$
and of degree $n=2$ and height at most $H$ is $O(H^2 \log H)$.
\end{lemma}

\subsection{An estimate of polynomials with roots close to the
roots of a real polynomial}

In the case of general (not monic) polynomials
the next lemma will be used in a number of estimates from below.

\begin{lemma}
\label{keylem}
Given a separable polynomial $$h(X)=b_0X^n+b_1X^{n-1}+\cdots+b_n \in \R[X]$$ of degree
$n \ge 1$
with roots $\be_1,\dots,\be_n$ and a sufficiently
small positive number $\varepsilon$,
there exists a constant $C=C(\varepsilon,h)$ such that the number
of integer polynomials $f$ of degree $n$ and height at most $H$ with
$n$ roots lying $\varepsilon$-neighbourhoods of $\be_1,\dots,\be_n$
respectively (and such that the root of each such $f$  lying in
the open disc $|z-\be_i|<\varepsilon$ is real if and only if
the number $\be_i$ is real) is at least $C H^{n+1}$.
\end{lemma}

\begin{proof}

Suppose that the polynomial $h(X)$ has exactly $r$ real roots $\be_1,\dots,\be_r$ and $2s$ non-real roots  $\be_{r+1},\dots,\be_n$. These occur as complex conjugate pairs $\be_{r+2i}=\overline{\be_{r+2i-1}}$ for $i=1,\dots,(n-r)/2$.
By Lemma~\ref{lem:conroot}, there exist two positive real numbers  $\ga, \eps$ depending only on $h$ such that any polynomial
\begin{equation}
\label{eq:type g1}
g(X)= a_0X^n+a_1X^{n-1}+\cdots + a_n \in \R[X]
\end{equation}
with $|a_i-b_i | < \eps $ for any $0\le i \le n$ has exactly one root in the disk $D_j$ with center at $\be_j$ and radius $\ga$ for each $1\le j \le n$.

For any $r+1\le j \le n$, since there are no real numbers in $D_j$, the root of $g$ located in $D_j$ must be non-real. Fix $j$ in the range $1\le j \le r$. If the root of $g$ located in $D_j$ is non-real, then its complex conjugate is also located in $D_j$, so $g$ has at least two distinct roots in $D_j$. This leads to a contradiction. Thus, $g$ has exactly $r$ real roots and $2s$ non-real roots.

For any $0\le i \le n$, denote by $I_i$ the open interval $(H(b_i-\eps),H(b_i+\eps))$. Consider a polynomial
\begin{equation}
\label{eq:type f1}
f(X)= a_0X^n+a_1X^{n-1}+\cdots + a_n \in \Z[X]
\end{equation}
with $a_i\in I_i$ for any $0\le i \le n$. Then $f(X)/H$ is the polynomial of type \eqref{eq:type g1}, so $f$ has exactly $r$ real roots and $2s$ non-real roots in the discs $|z-\be_j|<\gamma$, $j=1,\dots,n$. Note that there are at least $\lfloor 2\eps H \rfloor^{n+1}  \gg H^{n+1}$ of such integer polynomials $f$ which have the form \eqref{eq:type f1}. Their heights satisfy $H(f)\ll H$. For sufficiently
large $H$, since
all the implied constants depend on $h$, $\gamma$ and $\eps$
but are independent of $H$,  we deduce the lower bound
$CH^{n+1}$,
as claimed.
\end{proof}

\subsection{An estimate of polynomials with a multiplicative relation among its roots}
For further applications, we need to count a special kind of integer polynomials
whose roots satisfy some multiplicative relation.

\begin{lemma}
\label{lem:pair}
For integers $n\ge 4$ and $H\ge 1$, we have
\begin{align*}
|\{f\in \Z[X]: & \deg f= n, H(f)\le H,
\textrm{ $f$ is monic and has four roots } \\
& \textrm{ $\al_1,\al_2,\al_3,\al_4$ such that $\al_1\al_2=\al_3\al_4$} \}| \ll H^{n-1},
\end{align*}
and
\begin{align*}
|\{f\in \Z[X]: & \deg f= n, H(f)\le H,
\textrm{ $f$ has four roots $\al_1,\al_2, \al_3,\al_4$ } \\
& \textrm{such that $\al_1\al_2=\al_3\al_4$} \}| \ll H^n.
\end{align*}
\end{lemma}

\begin{proof}
The proofs of these two upper bounds are the same. Here, for brevity we only prove the second upper bound.
Write
$$
f(X)=  a_0X^n  + a_1X^{n-1} +\cdots + a_n = a_0\prod_{i=1}^{n}(X-\al_i) .
$$
Here, although $f$ may have multiple roots, we still view that all $\al_1,\ldots,\al_n$ are different (at least from the viewpoint of symbols). For integer $k$ in the range $1\le k \le n$, the $k$th elementary symmetric polynomial with respect to $\al_1,\ldots, \al_n$ is
$$
s_k(\al_1,\ldots,\al_n) = \sum_{1\le j_1 < j_2 < \cdots < j_k \le n} \al_{j_1}\al_{j_2}\cdots \al_{j_k}= (-1)^{k} \frac{a_k}{a_0}.
$$

Consider the polynomial
$$
 \prod_{1 \leq i < j \leq n} (X-\al_i \al_j).
$$
Its coefficients are symmetric polynomials with respect to $\al_1,\ldots, \al_n$ over $\Z$, so they are integer polynomials in the variables $a_1/a_0, \ldots, a_n/a_0$ (in fact, the integer coefficients depend only on $n$).
So, multiplied by an appropriate power of $a_0$ (depending only on $n$), this product is an integer polynomial, say $g(X)$,
whose coefficients are polynomials with integer coefficients in $a_0,a_1,\ldots ,a_n$. Let $\Delta$ be the discriminant of the polynomial $g$. Then, $\Delta$ is a polynomial in
$a_0,a_1, \ldots, a_n$ with integer coefficients (depending only on $n$).
The property $\al_1\al_2=\al_3\al_4$ implies that $\Delta=0$. So one of the coefficients $a_0,a_1,\ldots,a_n$,
say $a_0$, is predetermined in $m=\deg_{a_0} \Delta$ ways by other $n$ coefficients $a_1,\dots,a_n$ (note that
the degree of the polynomial $\Delta$, which is not less than $m$, depends  only on $n$). Consequently, the number
of such polynomials $f$ is $O(H^n)$.
\end{proof}

\subsection{Probabilistic Galois Theory}

In Section \ref{sec:modulus}, we will present an alternative proof of Theorem \ref{vienas}, by using Probabilistic Galois Theory. Here are some basic facts concerning it.

Let $S_n$ be the full symmetric group on a finite set of $n$ symbols.
Given a polynomial $f \in \Z[X]$, denote by $G_f$ the Galois group of the splitting field of $f$ over $\Q$.
For any integers $n\ge 1$ and $H\ge 1$, we define the sets
\begin{align*}
\E_n(H) =  \{f(X)= X^n + a_1X^{n-1} & +\cdots + a_n \in \Z[X]: \\
& G_f \ne S_n, |a_i| \le H, i=1,2,\cdots, n \},
\end{align*}
and
\begin{align*}
\E_n^*(H) =  \{f(X)=  a_0X^n & + a_1X^{n-1} +\cdots + a_n \in \Z[X]: \\
& G_f \ne S_n, a_0\ne 0, |a_i| \le H, i=0,1,\cdots, n \}.
\end{align*}
Then, put $E_n(H)=| \E_n(H) |$ and $E_n^*(H)=| \E_n^*(H) |$.

The study of bounding $E_n(H)$ and $E_n^*(H)$ was initiated by van der Waerden \cite{Waerden1936} who showed that
\begin{lemma}
\label{En}
For integers $n \ge 3$ and $H \ge 2$, we have
$$
E_n(H)\ll H^{n-1/(6(n-2)\log\log H)}
$$
and
$$
E_n^*(H) \ll H^{n+1-1/(6(n-2)\log\log H)},
$$
where the implied constants depend only on $n$.
\end{lemma}

Later, the bound on $E_n(H)$ has been improved in several ways; see \cite{Cohen,Dietmann2012,Dietmann2013,Gallagher}. For instance, Gallagher \cite{Gallagher} proved that $E_n(H) \ll H^{n-1/2}\log H$, and recently Zywina improved this bound to $E_n(H) \ll H^{n-1/2}$ in \cite[Proposition 1.5]{Zywina}.

\section{Counting integer polynomials via moduli of their roots}
\label{sec:modulus}

In \cite[Theorem 1.1]{DS2014b}, it has been shown that almost all monic integer polynomials are dominant (that is, have exactly one root with maximal modulus). Here, our considerations include both monic and general integer polynomials.

For positive integers $n,k$ with $k\le n$, we define the sets
\begin{align*}
\rA_n(k,H) =  \{&f(X)=  X^n  + a_1X^{n-1} +\cdots + a_n \in \Z[X]: H(f)\le H, \\
& \textrm{ $f$ has exactly $k$ roots with maximal modulus} \},
\end{align*}
and
\begin{align*}
\rA_n^*(k,H) =  \{&f(X)=  a_0X^n  + a_1X^{n-1} +\cdots + a_n \in \Z[X]: H(f)\le H, \\
& \textrm{$a_0\ne 0$, $f$ has exactly $k$ roots with maximal modulus} \}.
\end{align*}
Then, denote $A_n(k,H)= |\rA_n(k,H)|$ and $A_n^*(k,H)=
|\rA_n^*(k,H)|$.

In this section, we will estimate $A_n(k,H)$ and $A_n^*(k,H)$.

\subsection{The monic case}

It has been proved in \cite[Theorem 1.1]{DS2014b} that
$$
\lim_{H\to \infty} \frac{A_n(1,H)}{(2H+1)^n}=1.
$$
Now, we estimate the error term, that is, the difference between the quantities $A_n(1,H)$ and $(2H+1)^n$.

\begin{theorem}
\label{thm:moduli2+}
For integers $n\ge 2$ and $H\ge 1$, we have
$$
H^{n-1/2} \ll (2H+1)^n-A_n(1,H)= \sum_{k=2}^nA_n(k,H)\ll H^{n-1/2}.
$$
\end{theorem}

\begin{proof}
First, given an arbitrary polynomial $f(X)=X^n+a_1X^{n-1}+\cdots+a_n\in \C[X]$, let $R$ be the maximal modulus of the roots of $f$. Suppose that $f$ does not have exactly one root with maximal modulus, that is, it has at least two roots on the circle $|z|=R$. Then, by the definition of the Mahler measure, we have $R^2\le M(f)$. Noticing that $|a_1|\le nR$ and applying \eqref{Mahler}, we obtain
$$
|a_1| \le n(n+1)^{1/4}H(f)^{1/2},
$$
which implies that
$$
\sum_{k=2}^nA_n(k,H) \ll H^{n-1/2},
$$
since $|a_2|,\dots,|a_n| \leq H$.

For a lower bound, let us consider the following polynomials
 $$
 f(X)=X^n+a_1X^{n-1}+a_2X^{n-2}+\cdots+a_n\in \Z[X],
 $$
 where $-\delta \sqrt{H} \leq a_1<0$ for some fixed real number $0<\delta<1$ and $1\le a_i\le H$ for each $2\le i\le n$ and $a_2\ge \delta^2 H$,
and either $ a_2 \ge a_3, a_4\ge a_5, \ldots,  a_{n-2}\ge a_{n-1}$ if $n$ is even, or $ a_2 \ge a_3, a_4\ge a_5, \ldots,  a_{n-1}\ge a_n$ if $n$ is odd. Then, when $H$ is large enough, such polynomials $f$ indeed exist, and we further claim that $f\not\in \rA_n(1,H)$. It is clear that, by counting these polynomials $f$, we would get
$$
H^{n-1/2} \ll \sum_{k=2}^nA_n(k,H),
$$
which completes the proof.

To prove the claim it suffices to show that the largest in modulus root of $f$ is non-real. Assume that $f$ has a real  root $x\ge 0$. Clearly, $x>0$. Since
$a_2,\dots,a_n>0$, we have
$$
f(x) > x^{n-1}(x + a_1 ) \ge 0
$$
if $x \ge |a_1|$, so we must have $0<x < |a_1|$.
Thus, $|a_1x| < |a_1|^2 \leq \delta^2 H$.  Then
$$
f(x)>a_1x^{n-1}+a_2x^{n-2}=x^{n-2}(a_1x+a_2) > 0,
$$
which contradicts $f(x)=0$. So, $f$ does not have non-negative roots.

Set $g(X)=f(-X)$ if $n$ is even, otherwise let $g(X)=-f(-X)$. Then, the negative real roots of $f(X)$ are exactly the positive real roots of $g(X)$.
In case $n$ is even, as
$$
g(X)=X^n-a_1X^{n-1}+a_2X^{n-2}-a_3X^{n-3}+\cdots+a_{n-2}X^2-a_{n-1}X+a_n,
$$
for any real $x\ge 1$ we have
$$
g(x)> x^{n-3}(a_2x-a_3)+\cdots+x(a_{n-2}x-a_{n-1}) \ge 0.
$$
(Evidently, for $n=2$ the polynomial $g(X)$ has no positive roots.)
Similarly, if $n$ is odd, observe that
$$
g(X)=X^n-a_1X^{n-1}+a_2X^{n-2}-a_3X^{n-3}+\cdots+ a_{n-1}X-a_n,
$$
so for any real $x\ge 1$ we have
$$
g(x)> x^{n-3}(a_2x-a_3)+\cdots+(a_{n-1}x-a_n) \ge 0.
$$
Therefore, all the real roots of $f(X)$ (if any) lie in the interval $(-1,0)$.

Moreover, writing $f$ as $f(X)=\prod_{i=1}^{n}(X-\al_i)$, we obtain $a_n=(-1)^n\prod_{i=1}^{n}\al_i$.
Since $a_n\ge 1$, we see that $|\al_i|\ge 1$ for at least one index $i$. So, not all the roots of $f(X)$ lie strictly inside the unit circle. It follows that the largest
in modulus root of $f$ is non-real, and hence
$f\not\in \rA_n(1,H)$.
\end{proof}

We remark that for the polynomial $f$ constructed in the second part of the above proof, some extra conditions can be added to make sure that $f$ only has non-real roots when $n$ is even. For example, if $n\ge 4$ is even, we can further assume that
$$
\min_{1\le i \le n/2} a_{2i} \ge \max_{1\le i \le (n-2)/2} a_{2i+1}.
$$
Then, for any $0<x <1$, we have
\begin{align*}
g(x)&=x^n-a_1x^{n-1}+a_2x^{n-2}+ x^{n-4}(a_4 - a_3 x) + \cdots + (a_n -a _{n-1}x) \\
& > x^n-a_1x^{n-1}+a_2x^{n-2} >0,
\end{align*}
which implies that $f$ has no roots in the interval $(-1,0)$. Thus, $f$ indeed has no real roots.

The next theorem shows that the main contribution in the sum
$\sum_{k=2}^n A_n(k,H)$ considered in
Theorem \ref{thm:moduli2+} comes from $A_n(2,H)$.

\begin{theorem}
\label{thm:moduli3+}
For integers $n\ge 3$ and $H\ge 1$, we have
$$
\sum_{k=3}^nA_n(k,H) \ll H^{n-1}.
$$
\end{theorem}

\begin{proof}
We choose an arbitrary polynomial $f \in \bigcup_{k=3}^{n} \rA_n(k,H)$.
Suppose first that $f$ is reducible. Then, it follows directly from Lemma \ref{reducible} that the number of these polynomials $f$ is $O(H^{n-1})$.

Next, suppose that $f$ is irreducible.
If $f$ has a real root with maximal modulus (for instance, when $k$ is odd), then by Lemma~\ref{Ferguson}, we know that $f(X)=g(X^k)$ for some $g \in \Z[X]$. So,  the number of such polynomials $f$ is  bounded above by
$$
\sum_{k=3}^{n}(2H+1)^{n/k}\le (n-2)(2H+1)^{n/3} \ll H^{n-2}.
$$

Finally, assume that $f$ has no real roots with maximal modulus, but has at least three roots of maximal modulus.  Then, we must have $n\ge 4$, and so there are at least two pairs of complex conjugate (non-real) roots with maximal modulus, say, the pair $\al_1=\al, \al_2=\overline{\al}$ and another pair $\al_3=\be, \al_4=\overline{\be}$.
Clearly, we have
$$
\al_1\al_2=\al_3\al_4,
$$
and so the desired result follows from Lemma~\ref{lem:pair}.
\end{proof}

Combining Theorems \ref{thm:moduli2+} and \ref{thm:moduli3+}
we conclude that

\begin{corollary}
\label{cor:moduli2}
For integers $n\ge 2$ and $H\ge 1$, we have
$$
H^{n-1/2} \ll A_n(2,H) \ll H^{n-1/2}.
$$
\end{corollary}

\subsection{The non-monic case}

In \cite[Theorem 1.4]{DS2014b}, it has been shown that
for any integer $n\ge 2$ we have
$$
\liminf_{H\to \infty}\frac{A_n^*(1,H)}{(2H)^{n+1}}  \ge \frac{1}{3n^2 \sqrt{n+1}}
$$
and
\begin{equation}
\limsup_{H\to \infty}\frac{A_n^*(1,H)}{(2H)^{n+1}}\le
\left\{\begin{array}{ll}
 1-1/(3\cdot 2^{3n/2-1})&\quad\text{if $n$ is even},\\
 1-1/(3 \cdot 2^{(3n+1)/2}) &\quad\text{if $n$ is odd}.\\
\end{array}\right.
\notag
\end{equation}
The above two bounds say that the density of dominant integer polynomials  is positive but does not tend to $1$ as $H \to \infty$. Moreover, it has been conjectured in \cite{DS2014b} that
\begin{equation}
\label{re11}
\limsup_{H\to \infty}\frac{A_n^*(1,H)}{(2H)^{n+1}} > \frac{1}{2},
\end{equation}
which means that at least half of integer polynomials are dominant.

In this section, we show that almost all integer polynomials have exactly one or two roots with maximal modulus, and we also provide a strong evidence for the validity of \eqref{re11}.

The following result implies that the density of integer polynomials with more than two roots of maximal modulus tends to zero as $H \to \infty$.
(It can be viewed as a non-monic version of Theorem~\ref{thm:moduli3+}.)

\begin{theorem}
\label{thm:moduli3+*}
For integers $n\ge 3$ and $H\ge 1$, we have
$$
 \sum_{k=3}^nA_n^*(k,H) \ll H^n.
$$
\end{theorem}

\begin{proof}
Since the number of reducible integer polynomials of degree $n\ge 3$ and height at most $H$ is $O(H^n)$ (see, for example, Lemma \ref{reducible2}), we only need to consider those  polynomials contained in $\rA_n^*(k,H)$ that are irreducible. In the rest of the proof, let $f \in \rA_n^*(k,H)$ be an irreducible polynomial.

If $f$ has a real root with maximal modulus (for instance, when $k$ is odd), then by Lemma~\ref{Ferguson}, we know that $f(X)=g(X^k)$ for some $g\in \Z[X]$. So,  the number of such polynomials $f$ is
bounded above by
$$
\sum_{k=3}^{n}(2H+1)^{n/k+1} \le (n-2)(2H+1)^{n/3+1} \ll H^{n-1}.
$$

Now suppose that $f$ has no real roots with maximal modulus.
Then, we must have $n\ge 4$, and there are at least two pairs of
complex conjugate (non-real) roots with maximal modulus, say,
the pair $\al_1=\al, \al_2=\overline{\al}$ and another pair $\al_3=\be, \al_4=\overline{\be}$.
Now, as above we have
$
\al_1\al_2=\al_3\al_4,
$
and so the desired result follows from Lemma~\ref{lem:pair}.
\end{proof}

By considering the polynomials $a_0X^3+a_3$, we have $A_3^*(3,H) \gg H^2$. As an example of a lower bound we will also show that
$$
A_n^*(3,H) \gg H^{n-2}
$$
for any integer $n \ge 4$.
Indeed, consider the polynomials
$$
f(X)=(X^3+8)(a_0X^{n-3}+\cdots + a_{n-4}X + a_{n-3}) \in \Z[X],
$$
where $\lambda_1 H< a_0 < \lambda_2 H $, $\delta_1 H< a_i < \delta_2 H$ for $1\le i \le n-3$, and $0<\delta_1 < \delta_2 < \lambda_1 <\lambda_2 \le 1/9$. Clearly, $H(f)\le H$. Using
Lemma~\ref{Fujiwara}, we know that the moduli of all the roots of the polynomial $a_0X^{n-3}+\cdots + a_{n-4}X + a_{n-3}$ are less than 2. So, $f\in \rA_n^*(3,H)$.  Thus, counting those polynomials $f$ we obtain the lower bound $A_n^*(3,H) \gg H^{n-2}$.

Theorem~\ref{thm:moduli3+*} implies  Theorem~\ref{vienas} asserting that
almost all integer polynomials have exactly one or two roots with maximal modulus.
Indeed,
the number of integer polynomials of degree $n\ge 1$ and height at most $H \ge 1$ is exactly $2H(2H+1)^n=\sum_{k=1}^n A_n^*(k,H)$. Thus,  Theorem~\ref{vienas} follows directly from Theorem~\ref{thm:moduli3+*}.

In the sequel, we use the so-called Probabilistic Galois Theory to give an alternative proof of Theorem~\ref{vienas}, which, although gives weaker estimate on the error term, might be of some interest as well.
All we need to do is to show that for any $3 \le k \le n$, we have
$$
\lim_{H\to \infty} \frac{A_n^*(k,H)}{H^{n+1}} = 0.
$$

Now, assume that $3\le k \le n$. In fact, since the number of reducible integer polynomials of degree $n\ge 3$ and height at most $H$ is $O(H^n)$ (for example, see Lemma~\ref{reducible2}), we only need to consider those polynomials  contained in $\rA_n^*(k,H)$ that are irreducible.

Let $f\in \rA_n^*(k,H)$ be an irreducible polynomial.
If $f$ has a real root with maximal modulus, then by Lemma~\ref{Ferguson}, we know that $f(X)=g(X^k)$ for some $g\in \Z[X]$; so,  the number of such polynomials $f$ is $O(H^{n/k+1})$.

Now suppose that $f$ has no real roots with maximal modulus.  Then, there are at least two pairs of complex conjugate (non-real) roots with maximal modulus, say, the pair $\al_1=\al, \al_2=\overline{\al}$ and another pair $\al_3=\be, \al_4=\overline{\be}$. In particular, this yields $n \geq 4$.
Clearly, we have
\begin{equation}
\label{eq:pair}
\al_1\al_2=\al_3\al_4.
\end{equation}
We claim that $f\in \E_n^*(H)$, where $n \geq 4$.

Assume that $f \notin \E_n^*(H)$. Then, the Galois group of $f$ is $S_n$, so there exists an embedding $\sigma$ such that
$$
\sigma(\al_1)=\al_3,\, \sigma(\al_2)=\al_2,\, \sigma(\al_3)=\al_1,\, \sigma(\al_4)=\al_4.
$$
Together with \eqref{eq:pair}, we get $\al_3\al_2=\al_1\al_4$, and so $\al_1\al_3\al_2^2=\al_1\al_3\al_4^2$. Thus, we find that $\al_2^2=\al_4^2$, so $\al_2=-\al_4$. By a similar argument (interchanging $\al_1$ with $\al_4$), we find that $\al_2=-\al_3$. So, we obtain $\al_3=\al_4$, which is a contradiction. Thus, we must have $f \in \E_n^*(H)$.
Therefore, the desired result follows from Lemma~\ref{En}.

\section{Counting integer polynomials via the locations of their roots}
\label{sec:location}

In this section, we begin with the proof of Theorem~\ref{thm:rs11} asserting that integer polynomials with various locations of their roots have positive density.
It suffices to prove the lower bound.
But the lower bound follows from Lemma~\ref{keylem}, where
$$h(X)=(X-\be_1)\cdots (X-\be_n) \in \R[X]$$
is a separable monic polynomial,
combined with the second inequality of Lemma~\ref{En}. (For $n=2$, all irreducible polynomials have the Galois group $S_2$, so the result follows from Lemma~\ref{reducible2}.)

In conclusion,
we remark that, since it is likely that $A_n^*(1,H)>A_n^*(2,H)$, Theorem \ref{vienas} provides a strong evidence that the conjecture \eqref{re11} may be true. It also naturally implies the following approach.
Notice that for any polynomial having only non-zero roots, a root with maximal (resp. minimal) modulus corresponds to a root with minimal (resp. maximal) modulus of its reciprocal polynomial. So regarding Theorem~\ref{vienas}, we can only consider integer polynomials with exactly $m_1$ roots of minimal modulus and exactly $m_2$ roots of maximal modulus, where $m_1,m_2 \in \{1,2\}$. For integers $n\ge 1,H\ge 1$ and $m_1,m_2\in \{1,2\}$, let $\B_n^*(m_1,m_2;H)$ be the set of such integer polynomials of degree $n$ and height at most $H$, and set
$$
B_n^*(m_1,m_2;H)=|\B_n^*(m_1,m_2;H)|.
$$
Now, Theorem \ref{vienas} can be restated as
$$
\lim_{H \to \infty} \frac{B_n^*(1,1;H)+B_n^*(2,1;H)+B_n^*(1,2;H)+B_n^*(2,2;H)}{2H(2H+1)^{n}} = 1
$$
for any integer $n\ge 2$.

Using reciprocal polynomials, it is easy to see that $B_n^*(2,1;H)=B_n^*(1,2;H)$. So, if one could prove that $B_n^*(1,1;H)>B_n^*(2,2;H)$ (it is likely to be true), then the conjecture \eqref{re11} would be established.

For $n=3$ the conjecture is true, since $B_3^*(2,2;H)=0$.
But in general, it is not easy to compare $B_n^*(1,1;H)$ with $B_n^*(2,2;H)$. The following result implies that, for $n\ge 4$, the growth rate of each of the four quantities $B_n^*(m_1,m_2;H)$ is $H^{n+1}$ with respect to $H$.

\begin{theorem}
\label{thm:B11}
For integers $n\ge 3$ and $H\ge 1$, we have
\begin{align*}
& H^{n+1} \ll B_n^*(1,1;H) \ll H^{n+1},\\
& H^{n+1} \ll B_n^*(2,1;H) \ll H^{n+1}, \\
& H^{n+1} \ll B_n^*(1,2;H) \ll H^{n+1}.
\end{align*}
Furthermore, for integers $n\ge 4$ and $H\ge 1$, we have
\begin{align*}
 H^{n+1} \ll B_n^*(2,2;H) \ll H^{n+1}.
\end{align*}
\end{theorem}

\begin{proof}
All upper bounds are trivial.
The lower bounds follow from Theorem~\ref{thm:rs11}. Indeed, for instance, in order to get the last lower bound we can consider the polynomial $h(X)=(X^2+1)(X^2+5)(X^{n-4}+2^{n-4})$. The polynomials $f \in \Z[X]$ described in Lemma~\ref{keylem} will then be in the set $B_n^*(2,2;H)$, because their roots close to $\pm \sqrt{-1}$ will be conjugate complex numbers (with minimal modulus) and also
their roots close to $\pm  \sqrt{-5}$ will be conjugate complex numbers (with maximal modulus),
whereas
the remaining $n-4$ roots will have the moduli close to $2$.
\end{proof}

\section*{Acknowledgements}

The authors are very grateful to the
referee for useful comments and suggestions.  
They would like to thank Yuri Bilu for explaining the result in \cite{Waerden1936} to them. They also want to thank Igor E. Shparlinski and Manas Patra for useful discussions and helpful comments. The second author appreciates Wuxing Cai's valuable comments when he gave a talk on part of this paper at South China University of Technology.
The research of M.~S. was supported by the Australian Research Council Grant DP130100237.

\end{document}